\documentclass[12pt]{amsart}

\usepackage{enumitem}
\usepackage{psfrag}
\usepackage{graphicx}
\usepackage{pinlabel}
\usepackage{graphicx}
\usepackage{amsmath}
\usepackage{amssymb}
\usepackage{amscd}
\usepackage{picinpar}
\usepackage{times}
\usepackage{pb-diagram}
\usepackage{graphicx}
\usepackage{wrapfig}
\usepackage{xspace}
\usepackage{hyperref}
\usepackage{url}
\usepackage{caption}
\usepackage{subcaption}
\usepackage{fancyhdr}
\usepackage{overpic}
\usepackage{color}
\usepackage[square,comma,sort&compress,numbers]{natbib}
\usepackage{latexsym}
\usepackage{mathrsfs}
\usepackage{amsfonts}
\usepackage{hyperref}
\usepackage{amsthm}
\usepackage{appendix}
\usepackage{pdfsync}

\theoremstyle{definition}
\newtheorem{theorem}{Theorem}[section]
\newtheorem{lemma}[theorem]{Lemma}

\newtheorem{corollary}[theorem]{Corollary}
\newtheorem{definition}[theorem]{Definition}

\newtheorem{remark}[theorem]{Remark}
\newtheorem*{theorem*}{Theorem}
\def\qed{\hfill{Q.E.D.}\smallskip}

\begin{document}

\title{\bf The Calabi flow with prescribed curvature on finite graphs}
\author{Yi Li, Jie Wang, Pingsan Yuan, Chao Zheng}

\date{\today}

\address{School of Mathematical Sciences, University of Science and Technology of China, Hefei, 230026, P.R.China;}
\email{ustcyili@ustc.edu.cn}

\address{School of Mathematics and Statistics, Jiangsu Normal University, Xuzhou, 221116, P.R.China;}
\email{jiewang@jsnu.edu.cn}

\address{School of Mathematics and Statistics, Wuhan University, Wuhan, 430072, P.R.China;}
\email{pingsanyuan@whu.edu.cn}

\address{International Center for Mathematical Research (BICMR), Beijing (Peking) University, Beijing, 100871, P.R. China}
\email{chaozheng@pku.edu.cn}



\begin{abstract}
In this paper, we investigate the prescribed curvature problem associated with a special Lin-Lu-Yau curvature on finite graphs of girth at least 6.
We define the corresponding Calabi flow for this curvature type, and establish an equivalent characterization of the problem, namely, the solution to the Calabi flow exists globally in time and converges if and only if there exists a weight function that realizes the prescribed curvature.
In particular, for constant curvature weights, we prove that the solution to the Calabi flow exists globally in time and converges under certain topological conditions.
\end{abstract}

\maketitle

\section{introduction}

\subsection{Background}

Geometric flows have emerged as powerful tools in differential geometry for the study of canonical metrics with prescribed curvature on smooth manifolds.
The Ricci flow, first introduced by Hamilton \cite{Hamilton}, has revolutionized this field and paved the way for the resolution of several landmark problems, including the Poincar\'{e} conjecture \cite{Perelman1,Perelman2,Perelman3} and the sphere theorem \cite{BS}.
Motivated by the pursuit of constant curvature metrics, Calabi pioneered the variational analysis of the Calabi energy and introduced the celebrated Calabi flow \cite{Calabi1,Calabi2}. 
On closed Riemannian surfaces, the Calabi flow admits the following formulation
\begin{equation*}
\frac{\partial g}{\partial t} = \Delta_g K \cdot g,
\end{equation*}
where $g$ is the Riemannian metric, $K$ is the Gaussian curvature, and $\Delta_g$ is the Laplace-Beltrami operator  associated with the metric $g$.
Results concerning the longtime existence and convergence of the Calabi flow on closed surfaces were established in \cite{Chen,Chru}, and further references can be found in \cite{Chang1,Chang2,CLT}.

Polyhedral surfaces, regarded as discrete analogues of smooth Riemannian surfaces, serve as natural settings for defining combinatorial curvature flows. 
The notion of circle patterns was introduced by Thurston \cite{Thurston} as a foundational tool for studying hyperbolic structures on 3-manifolds. 
He constructed circle patterns equipped with prescribed intersection angles on triangulated surfaces.
The induced polyhedral metric generally admits conical singularities at vertices, which are characterized by the angle deficit curvatures. 
Motivated by Hamilton's work on the Ricci flow \cite{Hamilton},
Chow-Luo \cite{Chow-Luo} introduced the combinatorial Ricci flow on closed triangulated surfaces, 
which acts as a canonical discrete counterpart of Hamilton's smooth Ricci flow. 
Under suitable combinatorial assumptions, 
they proved that the combinatorial Ricci flow admits longtime existence and converges exponentially fast to Thurston's circle patterns on closed triangulated surfaces \cite{Chow-Luo}.
Since then, extensive research has been devoted to combinatorial curvature flows associated with various discrete conformal structures on surfaces. 
Typical developments in this direction include the combinatorial Yamabe flow \cite{L1}, the combinatorial Calabi flow \cite{Ge1,Ge2}, the fractional combinatorial Calabi flow \cite{Wu-Xu}, and the combinatorial $p$-th Calabi flow \cite{L-Z}, all of which are constructed for triangulated surfaces.

Furthermore, geometric flows have been successfully generalized to the discrete setting of graphs. 
In contrast to triangulations of closed compact surfaces, 
a graph consists merely of a vertex set and an edge set, devoid of any facial structure.
The discrete Ricci flow was first proposed by Ollivier \cite{Ollivier}, 
which governs the metric evolution driven by the Ollivier-Ricci curvature and has found broad applications across a variety of mathematical and applied disciplines. 
In recent years, extensive theoretical studies have been carried out on this flow, as well as its extensions to generalized Ollivier-type curvatures, over diverse families of graphs \cite{BHLL,BLLW,MY1,MY2,MY3,TMYZ}.
Notably, all these existing literature imposes a critical constraint: the graph distance function is determined by the edge weights. 
Nevertheless, this restriction is theoretically nonessential. 
For a general graph, the distance metric and edge weights can be defined independently, without being confined by such a fixed dependency.

In \cite{LL}, Lin-Liu investigated the evolutionary dynamics between graph edge weights and discrete curvature, while keeping the graph distance fixed throughout the flow. 
They introduced the Lin-Lu-Yau Ricci flow with prescribed curvature, and established an equivalent characterization connecting the global convergence of the flow to the existence of admissible weight functions. 
We will briefly review the core content of Lin-Liu's work in Subsection \ref{Subsec: LL work}.
In this paper, motivated by Ge's research on the combinatorial Calabi flow \cite{Ge1,Ge2}, Wu-Xu's work on the fractional combinatorial Calabi flow \cite{Wu-Xu}, and Lin-Zhang's study on the combinatorial $p$-th Calabi flow \cite{L-Z}, 
we introduce three classes of Calabi flows defined on graphs with girth at least $6$, and prove an analogous convergence result for these flows.

\subsection{Various Calabi flows on graphs}

Let $G=(V,E,\omega)$ be a finite weighted graph, 
where $V$ denotes the vertex set, $E \subset V \times V$ denotes the edge set, and $\omega\colon E \to \mathbb{R}_+$ is a positive weight function defined on the edge set. 
Let $E = \{e_1, \dots, e_n\}$ and write $\omega= (\omega_1, \omega_2, \dots, \omega_n)$ with $\omega_i = \omega(e_i)$ for each edge index $i$.
We use $\kappa$ to denote the Lin-Lu-Yau curvature, whose detailed definition is presented in Subsection \ref{Subsec: LLY curvature}. 
Accordingly, we write $\kappa= (\kappa_1, \kappa_2, \dots, \kappa_n)$ with $\kappa_i =\kappa(e_i)$ for each edge $e_i$. 
Formula (\ref{Eq: curvature}) reveals that the curvature $\kappa$ depends solely on the weight vector $\omega$. 
Via the parameter transformation $r_i = \ln \omega_i$ for all $i \in \{1, \dots, n\}$, 
the curvature $\kappa$ can be regarded as a function of the parameter vector $r$.

\begin{definition}
On a finite graph with girth at least $6$, the Calabi flow is defined by
\begin{equation}\label{Eq: CF1}
\begin{cases}
\dfrac{d}{dt} r_i = \Delta \kappa_i, \\
r_i(0) = r_{0,i},
\end{cases}
\end{equation}
where $r_i(0) = \ln\omega_i(0)$ is the initial value, and $\Delta$ is the discrete Laplace operator given by
\begin{equation}\label{Eq: Laplace}
\Delta g_i
= -\sum_{j=1}^n \frac{\partial \kappa_i}{\partial r_j}\, g_j
\end{equation}
for any function $g: E \to \mathbb{R}$.
\end{definition}

\begin{remark}
In analogy with the smooth setting, the Calabi flow (\ref{Eq: CF1}) is a negative gradient flow of the Calabi energy
\begin{equation*}
\mathcal{C}(r):=\frac{1}{2} \|\kappa(r)\|^2
=\frac{1}{2} \sum_{i=1}^n \kappa^2_i(r)
\end{equation*}
with respect to the standard Euclidean metric on $\mathbb{R}^n$. 
This geometric interpretation motivates both the definition of the Calabi flow (\ref{Eq: CF1}) and the discrete Laplace operator (\ref{Eq: Laplace}). Moreover, this formula is also used in \cite{LL}.
\end{remark}

For any prescribed curvature $\kappa^*$, 
the Calabi flow with prescribed curvature on a graph of girth at least $6$ is defined as
\begin{equation}\label{Eq: CF2}
\frac{d}{dt}r_i = \Delta (\kappa- \kappa^*)_i,
\end{equation}
which is called the modified Calabi flow in the following.
The associated Calabi energy is given by
\begin{equation}\label{Eq: energy}
\widetilde{\mathcal{C}}(r)
=\frac{1}{2} \|\kappa(r)-\kappa^*\|^2
=\frac{1}{2} \sum_{i=1}^n (\kappa_i(r)-\kappa^*)^2.
\end{equation}

Define the Jacobian matrix
\begin{equation*}
J=\frac{\partial(\kappa_1, \dots, \kappa_n)}{\partial(r_1, \dots, r_n)}
=\begin{pmatrix}
\frac{\partial \kappa_1}{\partial r_1} & \cdots & \frac{\partial \kappa_n}{\partial r_1} \\
\vdots & \ddots & \vdots \\
\frac{\partial \kappa_1}{\partial r_n} & \cdots & \frac{\partial \kappa_n}{\partial r_n}
\end{pmatrix}.
\end{equation*}
It follows directly from (\ref{Eq: Laplace}) that the discrete Laplace operator satisfies $\Delta=-J$. 
As proved in Lemma \ref{Lem: matrix}, the matrix $J$ is symmetric and positive semi-definite. 
By the spectral theorem for symmetric matrices, 
there exists an orthonormal matrix $Q\in \mathbb{R}^{n \times n}$ such that
\begin{equation*}
J = Q^\mathrm{T}\cdot \operatorname{diag}\{\lambda_1, \ldots, \lambda_n\}\cdot Q,
\end{equation*}
where $\lambda_1, \ldots, \lambda_n \geq 0$ are the eigenvalues of $J$. 
For any real parameter $s \in \mathbb{R}$, the $2s$-th order fractional discrete Laplace operator $\Delta^s$ is defined via spectral decomposition as
\begin{equation}\label{Eq: fractional Laplace}
\Delta^s = -J^s = -Q^\mathrm{T}\cdot \operatorname{diag}\{\lambda_1^s, \ldots, \lambda_n^s\}\cdot Q,
\end{equation}
with the convention that $0^s = 0$ for all $s \in \mathbb{R}$. 
Here $J^s$ denotes the matrix power of $J$. Accordingly, $\Delta^s$ is negative semi-definite and shares the same kernel as $\Delta$. 
In particular, setting $s=1$ recovers the discrete Laplace operator $\Delta=-J$.

\begin{definition}
For any $s \in \mathbb{R}$, the $2s$-th order fractional Calabi flow on a graph with girth at least $6$ is defined by 
\begin{equation}\label{Eq: CF-F}
\begin{cases}
\frac{d}{dt}r_i = \Delta^s (\kappa-\kappa^*)_i, \\
r_i(0) = r_{0,i},
\end{cases}
\end{equation}
where $\kappa^*$ is the prescribed curvature.
\end{definition}

\begin{remark}
The $2s$-th order fractional Calabi flow (\ref{Eq: CF-F}) reduces to the Lin-Lu-Yau Ricci flow introduced by Lin-Liu \cite{LL} when $s=0$, and recovers the modified Calabi flow (\ref{Eq: CF2}) when $s=1$.
\end{remark}

Lemma \ref{Lem: matrix} implies that $\sum_{j=1}^n \frac{\partial \kappa_i}{\partial r_j} = 0$ for all $i \in \{1,...,n\}$. 
As a consequence, the discrete Laplace operator $\Delta$ in (\ref{Eq: Laplace}) admits an equivalent reformulation in terms of adjacent edges:
\begin{equation*}
\Delta f_i
=-\sum_{j=1}^n \frac{\partial \kappa_i}{\partial r_j} f_j
=\sum_{j \sim i} \left( -\frac{\partial \kappa_i}{\partial r_j} \right) (f_j - f_i).
\end{equation*}
This alternative expression motivates the definition of a discrete $p$-Laplacian and the corresponding $p$-th Calabi flow.

\begin{definition}\label{Def: CF-P}
For any $p>1$, the $p$-th Calabi flow on a graph with girth at least $6$ is defined by
\begin{equation}\label{Eq: CF-P}
\begin{cases}
\frac{d}{dt}r_i= \Delta_{p} (\kappa-\kappa^*)_i, \\
r_i(0) = r_{0,i},
\end{cases}
\end{equation}
where $\Delta_{p}$ denotes the discrete $p$-Laplacian, acting on any function $f: E \to \mathbb{R}$ as
\begin{equation}\label{Eq: P-Laplace}
\Delta_{p} f_i = \sum_{j \sim i} \left( -\frac{\partial \kappa_i}{\partial r_j} \right) |f_j - f_i|^{p-2} (f_j - f_i).
\end{equation}
\end{definition}

\begin{remark}
When $p = 2$, the discrete $p$-Laplacian $\Delta_p$ defined in (\ref{Eq: P-Laplace}) reduces to the standard discrete Laplace operator $\Delta$ in (\ref{Eq: Laplace}), and the $p$-th Calabi flow (\ref{Eq: CF-P}) coincides exactly with the modified Calabi flow (\ref{Eq: CF2}).
\end{remark}

The main result of this paper is stated below, establishing longtime existence and exponential (or plain) convergence for solutions to the modified Calabi flow (\ref{Eq: CF2}), the fractional Calabi flow (\ref{Eq: CF-F}), and the $p$-th Calabi flow (\ref{Eq: CF-P}).

\begin{theorem}\label{Thm: main}
Let $G$ be a finite connected graph with girth at least $6$, and let $\kappa^*$ be a prescribed curvature satisfying the topological condition (\ref{Eq: GB}). Then the following statements are equivalent:
\begin{description}
\item[(i)] There exists a weight $r^*$ such that $\kappa(r^*) = \kappa^*$.

\item[(ii)] The solution $r(t)$ to the modified Calabi flow (\ref{Eq: CF2}) exists for all times and converges exponentially fast to $r^*$.

\item[(iii)]  For any $s\in \mathbb{R}$, the solution $r(t)$ to the $2s$-th order fractional Calabi flow (\ref{Eq: CF-F}) exists for all time and converges exponentially fast to $r^*$.

\item[(iv)]
For any $p>1$, the solution $r(t)$ to the $p$-th Calabi flow (\ref{Eq: CF-P}) exists for all time and converges to $r^*$.
\end{description}
\end{theorem}

\begin{remark}
When $s=0$, the convergence result for the $2s$-th order fractional Calabi flow in Theorem \ref{Thm: main} reduces to the corresponding result for the Lin-Lu-Yau Ricci flow; see Theorem \ref{Thm: LL} for details. 
In contrast to the modified Calabi flow (the case $p=2$), exponential convergence cannot be guaranteed for solutions to the $p$-th Calabi flow when $p \neq 2$.
\end{remark}

\subsection{Organization of the paper}

This paper is structured as follows. 
In Section \ref{Sec: PR}, we first recall the original definition of the Lin-Lu-Yau curvature. 
We then introduce the specialized Lin-Lu-Yau curvature adopted in this work, which is defined on graphs of girth at least $6$. 
Next, we review the seminal work of Lin-Liu on the Lin-Lu-Yau Ricci flow. 
Finally, we present several results concerning the existence of weight functions realizing constant prescribed curvatures. 
In Section \ref{Sec: proof}, we prove Theorem \ref{Thm: main}.

\section{Previous results}\label{Sec: PR}

\subsection{Lin-Lu-Yau curvature}\label{Subsec: LLY curvature}

Lin-Lu-Yau curvature is a modified version of the Ollivier–Ricci curvature \cite{Ollivier}, 
which is devised to measure the discrepancy between the Wasserstein distance and the intrinsic graph distance. In this section, we collect the necessary definitions and preliminary results. 
For more thorough treatments of this topic, we refer the readers to \cite{MW,Ollivier,LLY}.

Let $d(x,y)$ be the standard graph distance between vertices $x,y \in V$. 
For any two probability measures $\mu$ and $\nu$ on the vertex set $V$, 
the Wasserstein distance $W(\mu,\nu)$ is defined by
\begin{equation*}
W(\mu,\nu):=\inf_{\rho}\sum_{x,y\in V}\rho(x,y)d(x,y)
\end{equation*}
where the infimum ranges over all coupling functions $\rho:V \times V\to[0,1]$ satisfying the marginal constraints $\sum_{y\in V}\rho(x,y)=\mu(x)$ and $\sum_{x\in V}\rho(x,y)=\nu(y)$ for all $x,y\in V$.
For any $\alpha \in [0,1]$ and any vertex $x$, the probability measure $m_x^\alpha$ is defined as
\begin{equation*}
m_x^\alpha(y) = 
\begin{cases} 
\alpha & \text{if } y = x, \\
(1-\alpha)\dfrac{\omega_{xy}}{m(x)} & \text{if } y\sim x, \\
0 & \text{otherwise}.
\end{cases}
\end{equation*}
where $m(x) := \sum_{y \sim x} \omega_{xy}$ denotes the total weight of all edges incident to $x$. 
For any pair of distinct vertices $x,y\in V$, Lin-Lu-Yau curvature is given by
\begin{equation*}
\kappa(x, y) := \lim_{\alpha \to 1^-} \frac{1}{1 - \alpha} \left( 1 - \frac{W(m_x^\alpha, m_y^\alpha)}{d(x, y)} \right),
\end{equation*}
where $W(m_x^\alpha, m_y^\alpha)$ stands for the Wasserstein distance between $m_x^\alpha$ and $m_y^\alpha$.
Furthermore, M\"{u}nch-Wojciechowski \cite{MW} derived a limit-free characterization of the Ollivier–Ricci curvature. 
For any function $f:V \rightarrow \mathbb{R}$, we define the edgewise gradient
\begin{equation*}
\nabla_{xy} f := \frac{f(x) - f(y)}{d(x,y)}
\end{equation*}
for $x \neq y \in V$, and the corresponding Lipschitz norm
\begin{equation*}
\|\nabla f\|_\infty : = \sup_{x \sim y} |\nabla_{xy} f| \in [0,\infty].
\end{equation*}
For any constant $K \ge 0$, let $\text{Lip}(K) = \{ f \mid \|\nabla f\|_\infty \le K \}$ denote the space of $K$-Lipschitz functions on the graph. 
Under this notation, Lin-Lu-Yau curvature admits the equivalent variational formulation
\begin{equation}\label{Eq: curvature 2}
\kappa(x,y) = \inf_{\substack{f \in \text{Lip}(1) \\ \nabla_{yx} f = 1}} \nabla_{xy} \Delta f.
\end{equation}
where the graph Laplacian $\Delta$ is defined by
\begin{equation*}
\Delta f(x) = \frac{1}{m(x)} \sum_{z \sim x} \omega_{xz} (f(z) - f(x)), \quad \forall x \in V.
\end{equation*}

\subsection{Special Lin-Lu-Yau curvature}

In this work, we restrict our attention to finite connected graphs with girth at least $6$, 
which contain no cycles of length $3$, $4$, or $5$. 
Under this assumption, for an edge $e$ with endpoints $x$ and $y$, Lin-Lu-Yau curvature simplifies to
\begin{equation}\label{Eq: curvature}
\kappa_e = 2\omega_e \left( \frac{1}{m(x)} + \frac{1}{m(y)} \right) - 2.
\end{equation}

\begin{remark}
Formula \eqref{Eq: curvature} can be derived from \cite[Example 2.3]{MW} by substituting the function $f$ defined in \eqref{Eq: curvature 2} as
\begin{equation*}
f(z) =
\begin{cases}
0 & \text{if } z \sim x \text{ and } z \neq y, \\
1 & \text{if } z = x, \\
2 & \text{if } z = y, \\
3 & \text{if } z \sim y \text{ and } z \neq x.
\end{cases}
\end{equation*}
\end{remark}

Using \eqref{Eq: curvature} and elementary combinatorial arguments, Lin-Liu \cite{LL} gave the following identity
\begin{equation}\label{Eq: GB}
\sum_{e \in E} \kappa_e = 2\left(|V|-|E|\right).
\end{equation}
In fact,
\begin{equation*}
\begin{aligned}
\sum_{e \in E} \kappa_e
&=2 \sum_{e \in E} \omega_e \left( \frac{1}{m(x)} + \frac{1}{m(y)} \right) - 2|E|\\
&=2\sum_{z \in V} \sum_{e_z \in N(z)} \frac{\omega_{e_z}}{m(z)}- 2|E|\\
&=2\sum_{z \in V} 1-2|E|\\
&=2\left(|V|-|E|\right),
\end{aligned}
\end{equation*}
where $N(z)$ denotes the set of edges incident to the vertex $z$.

Furthermore, Lin-Liu \cite{LL} established the following rigidity result for the weight $\omega$ corresponding to the curvature $\kappa$.

\begin{lemma}[\cite{LL}, Proposition 3.2]
\label{Lem: rigidity}
Let $G = (V,E,\omega)$ be a graph with girth at least $6$. Then the weight function $\omega$ is uniquely determined by the curvature $\kappa$ up to a global scalar multiple.
\end{lemma}

\subsection{Lin-Lu-Yau Ricci flow}\label{Subsec: LL work}

To find the weight $\omega^*$ that realizes the prescribed curvature $\kappa^*$, 
Lin-Liu \cite{LL} introduced the following Lin-Lu-Yau Ricci flow with prescribed curvature:
\begin{equation}\label{Eq: RF}
\begin{cases}
\frac{d}{dt} \omega_i=(\kappa_i-\kappa^*)\omega_i, \\
\omega_i(0) = \omega_{0,i},
\end{cases}
\end{equation}
where $\omega(0)$ is a positive initial weight. 
They established the longtime existence and convergence of the solution to flow \eqref{Eq: RF}.

\begin{theorem}[\cite{LL}, Theorem 1.1]\label{Thm: LL}
On a graph with girth at least 6, the Lin-Lu-Yau Ricci flow (\ref{Eq: RF}) with a positive initial value converges exponentially to the weight of the prescribed curvature if and only if there exists $\omega^*$ such that $\kappa(\omega^*) =\kappa^*$.
\end{theorem}

By virtue of \eqref{Eq: GB}, the average curvature for graphs with girth at least $6$ is defined as
\begin{equation}\label{Eq: AC}
\overline{\kappa}=\frac{\sum_{e \in E} \kappa_e}{|E|} = 2\left( \frac{|V|}{|E|} - 1 \right).
\end{equation}

For constant curvature on graphs with girth at least $6$, Lin-Liu \cite{LL} derived the following necessary and sufficient condition for the existence of a constant curvature weight function.

\begin{theorem}[\cite{LL}, Theorem 3.1]\label{Thm: existence}
There exist positive weights on $E$ satisfying $\kappa_e \equiv \overline{\kappa}$ for any $e \in E$ if and only if
\begin{equation}\label{Eq: Key}
\max_{\emptyset \neq \Omega \subsetneq V} \frac{|E(\Omega)|}{|\Omega|} < \frac{|E|}{|V|},
\end{equation}
where $E(\Omega)$ denote the set of edges in the subgraph induced by $\Omega\subsetneq V$.
\end{theorem}

\subsection{Calabi flows with constant curvatures}

Substituting $\overline{\kappa}$ for $\kappa^*$ in the modified Calabi flow \eqref{Eq: CF2}, the $2s$-th order fractional Calabi flow \eqref{Eq: CF-F}, and the $p$-th Calabi flow \eqref{Eq: CF-P} yields the following three classes of Calabi flows with the constant curvature $\overline{\kappa}$:
\begin{equation}\label{Eq: F4}
\frac{d}{dt}r_i
=\Delta (\kappa- \overline{\kappa})_i
=\Delta \kappa_i-\Delta \overline{\kappa}_i
=\Delta \kappa_i,
\end{equation}
\begin{equation}\label{Eq: F5}
\begin{aligned}
\frac{d}{dt}r_i
=\Delta^s (\kappa-\overline{\kappa})_i
=\Delta^s \kappa_i-\Delta^s\overline{\kappa}_i
=\Delta^s \kappa_i,
\end{aligned}
\end{equation}
and
\begin{equation}\label{Eq: F6}
\frac{d}{dt}r_i
= \Delta_{p} (\kappa-\overline{\kappa})_i.
\end{equation}

Note that Theorem \ref{Thm: main} remains valid if we replace $\kappa^*$ with $\overline{\kappa}$ in statement (i) of Theorem \ref{Thm: main}. 
Combining this with Theorem \ref{Thm: existence}, we obtain the following corollary.
\begin{corollary}
The solutions to these three Calabi flows \eqref{Eq: F4}, \eqref{Eq: F5} and \eqref{Eq: F6} converge to the weight $\omega$ with constant curvature $\overline{\kappa}$ if and only if condition \eqref{Eq: Key} is satisfied.
\end{corollary}

\section{Proof of Theorem \ref{Thm: main}}\label{Sec: proof}

\subsection{Some useful lemmas}

The following lemma establishes key properties of the Jacobian matrix $J$. 
For completeness, we supply a detailed proof.
\begin{lemma}[\cite{LL}]\label{Lem: matrix}
The matrix $J=\dfrac{\partial(\kappa_1,\dots,\kappa_n)}{\partial(r_1,\dots,r_n)}$ is symmetric, positive semi-definite, and of rank $n-1$ with kernel $\{ c\mathbf{1}^\mathrm{T}\in \mathbb{R}^n \mid c \in \mathbb{R} \}$.
\end{lemma}
\proof
Consider two distinct edges $e_i$ and $e_j$ with $j\neq i$. 
Let $x,y$ denote the two endpoints of $e_i$. Differentiating \eqref{Eq: curvature} and recalling $m(x)=\sum_{y\sim x}\omega_{xy}$, we obtain
\begin{equation}\label{Eq: F3}
\frac{\partial \kappa_i}{\partial r_j}
=
\begin{cases}
-\dfrac{2\omega_i \omega_j}{m^2(x)}, & e_j \text{ is incident to } x,\\[6pt]
-\dfrac{2\omega_i \omega_j}{m^2(y)}, & e_j \text{ is incident to } y,\\[4pt]
0, & \text{otherwise}.
\end{cases}
\end{equation}
Furthermore, direct computation yields
\begin{equation*}
\begin{aligned}
\frac{\partial \kappa_i}{\partial r_i}
&= \omega_i\left[2\left(\frac{1}{m(x)}+\frac{1}{m(y)}\right)
-2\omega_i\left(\frac{1}{m^2(x)}+\frac{1}{m^2(y)}\right)\right]\\
&= \sum_{\substack{e_j\in E(x)\\ j\neq i}}\frac{2\omega_i\omega_j}{m^2(x)}
+\sum_{\substack{e_j\in E(y)\\ j\neq i}}\frac{2\omega_i\omega_j}{m^2(y)},
\end{aligned}
\end{equation*}
where $E(u)$ denotes the set of all edges incident to the vertex $u$.
Consequently, the matrix $\bigl(\frac{\partial\kappa_i}{\partial r_j}\bigr)_{n\times n}$ is symmetric, satisfies
\begin{equation*}
\frac{\partial\kappa_i}{\partial r_i}>0,\quad
\frac{\partial\kappa_i}{\partial r_j}\le 0\;\text{ for all }i\neq j,
\end{equation*}
and obeys the condition $\sum_{i=1}^n\frac{\partial\kappa_i}{\partial r_j}=0$ for every $j$.
The desired conclusion then follows from a standard linear algebraic fact (see \cite[Lemma~3.10]{Chow-Luo} for a proof):\\
\noindent\textbf{Lemma:}
Suppose $A = [a_{ij}]_{n \times n}$ is a symmetric matrix.
If $a_{ii} > 0$ and $a_{ij} \leq 0$ for all $i \neq j$ so that $\sum_{i=1}^n a_{ij} = 0$ for all $j$,
then $A$ is semi-positive definite so that its kernel is one-dimensional.

\qed

Lemma \ref{Lem: matrix} implies the following invariance result.
\begin{lemma}\label{Lem: invariant}
Let $\kappa^*$ be a prescribed curvature satisfying \eqref{Eq: GB}. 
Then the sum $\sum_{i=1}^n r_i(t)$ is invariant along the modified Calabi flow \eqref{Eq: CF2}, the $2s$-th order fractional Calabi flow \eqref{Eq: CF-F}, and the $p$-th Calabi flow \eqref{Eq: CF-P}.
\end{lemma}
\proof
First, along the modified Calabi flow \eqref{Eq: CF2}, a direct computation gives
\begin{equation*}
\frac{d}{dt}\left( \sum_{i=1}^n r_i \right)
= \sum_{i=1}^n \frac{dr_i}{dt}
= \sum_{i=1}^n \Delta (\kappa - \kappa^*)_i
= -\sum_{i=1}^n \sum_{j=1}^n J (\kappa - \kappa^*)_j
= 0,
\end{equation*}
where the last equality is a direct consequence of Lemma \ref{Lem: matrix}. 
This proves that $\sum_{i=1}^n r_i$ is invariant along the modified Calabi flow \eqref{Eq: CF2}.

Next, along the $2s$-th order fractional Calabi flow \eqref{Eq: CF-F}, we obtain
\begin{equation*}
\frac{d}{dt}\left( \sum_{i=1}^n r_i \right)
= \sum_{i=1}^n \Delta^s (\kappa - \kappa^*)_i
= - \mathbf{1}^\mathrm{T} J^s (\kappa - \kappa^*)
= 0,
\end{equation*}
where the final equality holds by Lemma \ref{Lem: matrix}. 
Thus, $\sum_{i=1}^n r_i$ is invariant along the $2s$-th order fractional Calabi flow \eqref{Eq: CF-F}.

Finally, along the $p$-th Calabi flow \eqref{Eq: CF-P}, we have
\begin{equation*}
\frac{d}{dt}\left( \sum_{i=1}^n r_i \right)
= \sum_{i=1}^n \Delta_{p} (\kappa - \kappa^*)_i.
\end{equation*}
By the definition of the discrete $p$-Laplacian $\Delta_{p}$ in \eqref{Eq: P-Laplace} and the symmetry of the matrix $\bigl(\frac{\partial \kappa_i}{\partial r_j}\bigr)_{n\times n}$, for any function $g: E \to \mathbb{R}$, the identity holds:
\begin{equation}\label{Eq: F1}
\begin{aligned}
\sum_{i=1}^n \Delta_{p} g_i
&=\sum_{i=1}^n \sum_{j \sim i} \left( -\frac{\partial \kappa_i}{\partial r_j} \right) |g_j - g_i|^{p-2} (g_j - g_i)\\
&=\frac{1}{2} \sum_{i=1}^n \sum_{j \sim i} \left( -\frac{\partial \kappa_i}{\partial r_j} \right)|g_j - g_i|^{p-2} (g_j - g_i)\\
&\quad + \frac{1}{2} \sum_{i=1}^n \sum_{j \sim i} \left( -\frac{\partial \kappa_i}{\partial r_j} \right)|g_j - g_i|^{p-2} (g_i - g_j)\\
&=0.
\end{aligned}
\end{equation}
Substituting $g = \kappa - \kappa^*$ into \eqref{Eq: F1} yields $\frac{d}{dt}\left( \sum_{i=1}^n r_i \right) = 0$ along the $p$-th Calabi flow \eqref{Eq: CF-P}. 
Therefore, $\sum_{i=1}^n r_i$ is invariant along this flow.
\qed

Lemma \ref{Lem: invariant} guarantees that all solutions to the modified Calabi flow \eqref{Eq: CF2}, the $2s$-th order fractional Calabi flow \eqref{Eq: CF-F}, and the $p$-th Calabi flow \eqref{Eq: CF-P} remain confined to the affine hyperplane
\begin{equation*}
P = \biggl\{ r \in \mathbb{R}^n \,\bigg|\, \sum_{i=1}^n r_i = \sum_{i=1}^n r_{i,0} \biggr\}.
\end{equation*}

The next lemma describes a growth property of convex functions, and a detailed proof can be found in \cite[Lemma 4.6]{GX JFA}.
\begin{lemma}\label{Lem: proper}
Let $f: \mathbb{R}^n \to \mathbb{R}$ be a $C^1$-smooth convex function such that $\nabla f(x_0) = 0$ for some $x_0 \in \mathbb{R}^n$.
If $f$ is $C^2$-smooth and strictly convex in a neighborhood of $x_0$, then $\lim_{\|x\| \to +\infty} f(x) = +\infty$.
\end{lemma}

\subsection{Proof of Theorem \ref{Thm: main}}

We split Theorem \ref{Thm: main} into three separate statements and prove each one individually.

\begin{theorem}\label{Thm: main1}
In Theorem \ref{Thm: main}, the statement \textup{(i)} is equivalent to the statement \textup{(ii)}.
\end{theorem}
\proof
\noindent\textbf{(ii)\,$\Rightarrow$\,(i):}
Suppose the solution $r(t)$ to the modified Calabi flow \eqref{Eq: CF2} converges to $r^*$ as $t\to\infty$. 
By continuity of the curvature $\kappa$,
\begin{equation*}
\kappa(r^*)
=\kappa\!\left(\lim_{t\to\infty}r(t)\right)
=\lim_{t\to\infty}\kappa(r(t)).
\end{equation*}
Applying the Mean Value Theorem, there exists a sequence $t_n\in(n,n+1)$ such that for every $i$,
\begin{equation*}
r_i(n+1)-r_i(n)=r_i'(t_n)=\Delta\bigl(\kappa(r(t_n))-\kappa^*\bigr)_i\to 0,\quad \text{as } n\to\infty.
\end{equation*}
This implies $\kappa(r^*)-\kappa^*=\lim_{n\to\infty}\bigl(\kappa(r(t_n))-\kappa^*\bigr)$ belongs to the kernel of $\Delta$. 
By Lemma \ref{Lem: matrix}, we have $\kappa(r^*)-\kappa^*=c\mathbf{1}^{\mathrm{T}}$ for some constant $c\in\mathbb{R}$. 
Since $\kappa^*$ satisfies \eqref{Eq: GB},
\begin{equation*}
\sum_{i=1}^n\bigl(\kappa(r^*)-\kappa^*\bigr)
=2(|V|-|E|)-2(|V|-|E|)=0,
\end{equation*}
which forces $c=0$. 
Hence $\kappa(r^*)=\kappa^*$, so $r^*$ is a weight with the prescribed curvature $\kappa^*$.

\noindent\textbf{(i)\,$\Rightarrow$\,(ii):}
Assume there exists a weight $r^*$ such that $\kappa(r^*)=\kappa^*$. 
Define the function $f:P\to\mathbb{R}$ by
\begin{equation}\label{Eq: function}
f(r)=\int_{r^*}^{r}\sum_{i=1}^n\bigl(\kappa_i-\kappa_i^*\bigr)\,dr_i.
\end{equation}
One directly verifies $f(r^*)=0$ and $\nabla_r f(r^*)=0$. Moreover,
$\operatorname{Hess}_r f=J$. 
By Lemma \ref{Lem: matrix},
$\operatorname{Hess}_r f$ is positive semi-definite with kernel orthogonal to the affine hyperplane $P$. 
By Lemma \ref{Lem: proper}, we obtain $\lim_{\|r\|\to\infty}f(r)\big|_{P}=+\infty$.
Thus $f|_{P}$ is a proper convex function satisfying $0=f(r^*)\le f(r)$ for all $r\in P$.

Along the modified Calabi flow \eqref{Eq: CF2}, we have
\begin{equation*}
\frac{d}{dt}f(r(t))
=\sum_{i=1}^n\frac{\partial f}{\partial r_i}\frac{dr_i}{dt}
=\sum_{i=1}^n(\kappa-\kappa^*)_i\Delta(\kappa-\kappa^*)_i
=-(\kappa-\kappa^*)^{\mathrm{T}}J(\kappa-\kappa^*)\le 0.
\end{equation*}
Hence $f(r(t))$ is non-increasing, and $0\le f(r(t))\le f(r(0))$ for all $t\ge 0$. 
By Lemma \ref{Lem: invariant}, 
the solution $r(t)$ is confined to the affine hyperplane $P$.
Combined with the properness of $f|_P$, $\{r(t)\}_{t \geq 0}$ is contained in a compact subset of $P$.
This implies the solution to the modified Calabi flow (\ref{Eq: CF2}) exists for all time.

By Lemma \ref{Lem: matrix}, $J$ is positive semi-definite of rank $n-1$ with kernel $\{c\mathbf{1}^\mathrm{T} \mid c \in \mathbb{R}\}$, which is orthogonal to $P$.
Consequently, $J^2$ inherits this structure and restricts to a strictly positive definite matrix on $P$. 
By continuity of eigenvalues for matrix-valued functions, all non-zero eigenvalues of $J^2(r)$ admit a uniform lower bound $\lambda_0>0$ over the compact set containing $\{r(t)\}$.
Now consider the Calabi energy \eqref{Eq: energy}. 
Direct calculations yield
\begin{equation*}
\frac{d}{dt}\widetilde{\mathcal{C}}(r(t))
=\sum_{i=1}^n\frac{\partial\widetilde{\mathcal{C}}}{\partial r_i}\frac{dr_i}{dt}
=-(\kappa-\kappa^*)^{\mathrm{T}}J^2(\kappa-\kappa^*)
\le -\lambda_0\widetilde{\mathcal{C}}(t).
\end{equation*}
This yields
\begin{equation*}
\widetilde{\mathcal{C}}(t)\le\widetilde{\mathcal{C}}(0)\,e^{-\lambda_0 t},\quad \forall t\ge 0.
\end{equation*}
Finally, by Lemma \ref{Lem: rigidity} and Lemma \ref{Lem: matrix}, the restriction $\kappa|_{P}:P\to\kappa(P)$ is a $C^1$-diffeomorphism. The inverse function theorem implies $\kappa^{-1}$ is locally Lipschitz continuous on $\kappa(P)$.
The compactness of $\{r(t)\}$ further provides constants $C_1,C_2>0$ such that
\begin{equation*}
\|r(t)-r^*\|^2
\le C_1\|\kappa(r(t))-\kappa(r^*)\|^2
=2C_1\widetilde{\mathcal{C}}(t)
\le C_2 e^{-\lambda_0 t}.
\end{equation*}
This establishes exponential convergence for the modified Calabi flow \eqref{Eq: CF2}. 
\qed

\begin{theorem}\label{Thm: main2}
In Theorem \ref{Thm: main}, the statement \textup{(i)} is equivalent to the statement \textup{(iii)}.
\end{theorem}
\proof
\noindent\textbf{(iii)\,$\Rightarrow$\,(i):}
Combining \eqref{Eq: fractional Laplace} with Lemma \ref{Lem: matrix}, 
the matrix power $J^{s}$ shares the same spectral properties as $J$ for any $s\in\mathbb{R}$: 
it is positive semi-definite of rank $n-1$, with kernel $\{c\mathbf{1}^\mathrm{T} \mid c \in \mathbb{R}\}$, which is orthogonal to the affine hyperplane $P$. 
The argument is essentially identical to that in Theorem \ref{Thm: main1}, so we omit the detailed proof here.

\noindent\textbf{(i)\,$\Rightarrow$\,(iii):}
By the arguments similar to that in Theorem \ref{Thm: main1},
the function $f(r)$ defined by (\ref{Eq: function}) can also be applied to this theorem.
Along the fractional Calabi flow (\ref{Eq: CF-F}), we  have
\begin{equation*}
\frac{d}{dt}f(r(t))
=\sum_{i=1}^n \frac{\partial f}{\partial r_i}\frac{dr_i}{dt}
=\sum_{i=1}^n \bigl(\kappa-\kappa^*\bigr)_i\,\Delta^s\bigl(\kappa-\kappa^*\bigr)_i
=-\bigl(\kappa-\kappa^*\bigr)^{\mathrm{T}}J^{s}\bigl(\kappa-\kappa^*\bigr)\le 0.
\end{equation*}
Hence $0\le f(r(t))\le f(r(0))$ for all $t\ge 0$. 
By Lemma \ref{Lem: invariant} and the properness of $f|_{P}$, 
we conclude that $\{r(t)\}_{t\ge 0}$ remains within a compact subset of $P$. 
This implies the solution of the fractional Calabi flow (\ref{Eq: CF-F}) exists for all time and $f(r(t))|_P$ converges as $t \to +\infty$.

Applying the Mean Value Theorem, there exists a sequence $t_n\in(n,n+1)$ such that
\begin{equation*}
f(r(n+1))-f(r(n))
=\frac{d}{dt}f(r(t))\bigg|_{t=t_n}
=-\bigl(\kappa(r(t_n))-\kappa^*\bigr)^{\mathrm{T}}J^{s}\bigl(\kappa(r(t_n))-\kappa^*\bigr)\to 0,\, \text{as}\, n\to\infty.
\end{equation*}
By Lemma \ref{Lem: matrix},
the above limit implies $\lim_{n \to +\infty} (\kappa(r(t_n))-\kappa^*)$ lies in the kernel of $J^s$.
Hence, there exists a constant $c \in \mathbb{R}$ such that
\begin{equation*}
\lim_{n \to +\infty} (\kappa(r(t_n))-\kappa^*) = c\mathbf{1}^\mathrm{T}.
\end{equation*}
Since $\kappa^*$ satisfies (\ref{Eq: GB}), we compute
\begin{equation*}
\sum_{i=1}^N (\kappa(r(t_n))-\kappa^*)
= 2(|V| - |E|)-2(|V| - |E|) =0,
\end{equation*}
which forces $c=0$. 
Consequently, $\lim_{n\to\infty}\kappa(r(t_n))=\kappa^*$.
As $\{r(t_n)\}$ lies in a compact subset of $P$, we may extract a convergent subsequence, still denoted $\{r(t_n)\}$, such that $r(t_n)\to\bar{r}\in P$. By continuity of the curvature $\kappa$,
\begin{equation*}
\kappa(\bar{r})
=\lim_{n\to\infty}\kappa(r(t_n))
=\kappa^*=\kappa(r^*).
\end{equation*}
Lemma \ref{Lem: rigidity} then yields $\bar{r}=r^*$, so $r(t_n)\to r^*$.

To establish the global convergence of the solution $r(t)$ to $r^*$, define the vector field $\Gamma(r) = \Delta^s(\kappa(r)-\kappa^*)$.
The Fr\'{e}chet derivative of $\Gamma$ at $r^*$ is $D\Gamma|_{r=r^*}=-J^{s+1}$.
Restricting this derivative to the affine hyperplane $P$, $D\Gamma|_{r=r^*}$ becomes strictly negative definite.
Therefore, $r^*$ is a local attractor for the fractional Calabi flow \eqref{Eq: CF-F}. 
The desired global convergence follows from the standard Lyapunov stability theory; see, for instance, \cite[Chapter 5]{Pontryagin}.
\qed

\begin{remark}
The exponential convergence of solutions to the fractional Calabi flow \eqref{Eq: CF-F} can be established using the Calabi energy \eqref{Eq: energy}, via an argument analogous to that used in the proof of Theorem \ref{Thm: main1}.
\end{remark}

\begin{theorem}\label{Thm: main3}
In Theorem \ref{Thm: main}, the statement \textup{(i)} is equivalent to the statement \textup{(iv)}.
\end{theorem}
\proof
\noindent\textbf{(iv)\,$\Rightarrow$\,(i):}
Suppose that the solution $r(t)$ to the $p$-th Calabi flow \eqref{Eq: CF-P} converges to $r^*$ as $t\to +\infty$. 
By the continuity of the curvature $\kappa$, 
we have $\kappa(r^*)=\lim_{t\to\infty}\kappa(r(t))$. Moreover, there exists a sequence $t_n\in(n,n+1)$ such that for every $i$,
\begin{equation*}
r_i(n+1) - r_i(n)
= r'_i(t_n)
= \Delta_p \left(\kappa(r(t_n)) - \kappa^* \right)_i
\to 0,
\quad \text{as } n \to +\infty.
\end{equation*}
Define $\widetilde{\kappa}
=\lim_{n\rightarrow +\infty}(\kappa(r(t_n))-\kappa^*)
=\kappa(r^*)-\kappa^*$.
which immediately gives $\Delta_p \widetilde{\kappa}=0$. 
Analogous to \eqref{Eq: F1}, for any function $g:E\to\mathbb{R}$, the following identity holds:
\begin{equation}\label{Eq: F2}
\begin{aligned}
\sum_{i=1}^n g_i\Delta_{p} g_i
&=\sum_{i=1}^n \sum_{j \sim i} \left( -\frac{\partial \kappa_i}{\partial r_j} \right) |g_j - g_i|^{p-2} (g_j - g_i)g_i\\
&=\frac{1}{2} \sum_{i=1}^n \sum_{j \sim i} \left( -\frac{\partial \kappa_i}{\partial r_j} \right)|g_j - g_i|^{p-2} (g_j - g_i)g_i\\
&\ \ \ \ + \frac{1}{2} \sum_{i=1}^n \sum_{j \sim i} \left( -\frac{\partial \kappa_i}{\partial r_j} \right)|g_j - g_i|^{p-2} (g_i - g_j)g_j\\
&=\frac{1}{2} \sum_{i=1}^n \sum_{j \sim i} \left( \frac{\partial \kappa_i}{\partial r_j} \right)|g_j - g_i|^{p}.
\end{aligned}
\end{equation}
Substituting $g=\widetilde{\kappa}$ into (\ref{Eq: F2}) yields
\begin{equation*}
0=\widetilde{\kappa}^\mathrm{T} \Delta_{p} \widetilde{\kappa}
=\sum_{i=1}^n \widetilde{\kappa}_i \Delta_{p} \widetilde{\kappa}_i
= \frac{1}{2} \sum_{i=1}^n \sum_{j \sim i} \frac{\partial \kappa_i}{\partial r_j} |\widetilde{\kappa}_j - \widetilde{\kappa}_i|^p.
\end{equation*}
From \eqref{Eq: F3}, we know that $\frac{\partial \kappa_i}{\partial r_j}<0$ for all adjacent edges $j\sim i$, 
which forces $\widetilde{\kappa}_i = \widetilde{\kappa}_j$. 
By the connectedness of the underlying graph, $\widetilde{\kappa}$ must be a constant function, meaning $\widetilde{\kappa} \equiv c$ for some constant $c\in\mathbb{R}$, or equivalently, $\kappa(r^*) -\kappa^* = c\mathbf{1}^\mathrm{T}$. 
Since $\kappa^*$ satisfies \eqref{Eq: GB}, a direct computation gives
\begin{equation*}
\sum_{i=1}^n \left(\kappa(r^*)-\kappa^*\right)
= 2(|V| - |E|)-2(|V| - |E|)= 0,
\end{equation*}
which implies $c=0$. 
Therefore, $\kappa(r^*)=\kappa^*$.

\noindent\textbf{(i)\,$\Rightarrow$\,(iv):}
We employ the same auxiliary function $f(r)$ defined in \eqref{Eq: function}, following arguments parallel to those in Theorem \ref{Thm: main1}. 
Along the $p$-th Calabi flow \eqref{Eq: CF-P}, we compute
\begin{equation*}
\begin{aligned}
\frac{df(r(t))}{dt}
&= \sum_{i=1}^n \frac{\partial f}{\partial r_i} \cdot \frac{dr_i}{dt} \\
&= \sum_{i=1}^n (\kappa-\kappa^*)_i \Delta_p(\kappa-\kappa^*)_i \\
&= \frac{1}{2} \sum_{i=1}^n \sum_{j \sim i} (\frac{\partial \kappa_i}{\partial r_j}) \left| (\kappa-\kappa^*)_j - (\kappa-\kappa^*)_i \right|^p\\
&\leq 0,
\end{aligned}
\end{equation*}
where the third line uses \eqref{Eq: F2}. 
This implies that $0 \leq f(r(t)) \leq f(r(0))$ for all $t\geq 0$. 
Combining Lemma \ref{Lem: invariant} and the properness of $f|_P$, 
we deduce that $\{r(t)\}_{t\geq 0}$ is contained in a compact subset of $P$. 
As a consequence, the solution to the $p$-th Calabi flow \eqref{Eq: CF-P} exists for all time, and $f(r(t))|_P$ converges as $t\to\infty$.

Furthermore, there exists a sequence $t_n\in(n,n+1)$ such that as $n \to +\infty$,
\begin{equation*}
\begin{aligned}
f(r(n+1)) - f(r(n))
&= \frac{df(r(t))}{dt}\bigg|_{t=t_n} \\
&= \frac{1}{2} \sum_{i=1}^n \sum_{j \sim i} (\frac{\partial \kappa_i}{\partial r_j}) \left| (\kappa(r(t_n)) - \kappa^*)_j - (\kappa(r(t_n)) - \kappa^*)_i \right|^p \\
&\to 0.
\end{aligned}
\end{equation*}
Repeating the arguments used in the part ``(iv)\,$\Rightarrow$\,(i)'',
we obtain $\lim_{n \to +\infty} \kappa(r(t_n)) = \kappa^*=\kappa(r^*)$.
Given that $\{r(t)\} \subset \subset P$,
we can extract a convergent subsequence $\{r(t_{n_k})\}$ such that $\lim_{k\to\infty}r(t_{n_k})=\overline{r}\in P$.
By continuity of $\kappa$, we have $\kappa(\overline{r}) = \lim_{k \to +\infty} \kappa(r(t_{n_k})) = \kappa^*$.
Applying Lemma \ref{Lem: rigidity}, 
we conclude that $\overline{r}=r^*$, so $\lim_{k\to\infty}r(t_{n_k})=r^*$.

It remains to prove the global convergence of $r(t)$ to $r^*$. We proceed by contradiction: suppose that $\lim_{t\to\infty}r(t)\neq r^*$. Then there exist a constant $\delta>0$ and a sequence $\xi_n\to\infty$ such that $\|r(\xi_n)-r^*\|>\delta$ for all $n$. Let $B_{P}(r^*,\delta)=P\cap B(r^*,\delta)$ denote the open ball centered at $r^*$ with radius $\delta$ restricted to $P$, so that $\{r(\xi_n)\}\subseteq P\setminus\overline{B_{P}(r^*,\delta)}$. Since $f|_P$ is proper and convex with a unique global minimum at $r^*$, $f$ attains a positive lower bound on the compact set $P\setminus B_{P}(r^*,\delta)$. That is, there exists a constant $C>0$ such that $f(r)\geq C$ for all $r\in P\setminus B_{P}(r^*,\delta)$, which implies $f(r(\xi_n))\geq C>0$ for all $n$. However, $f(r(t))$ is convergent and $\lim_{k\to\infty}f(r(t_{n_k}))=f(r^*)=0$, so $\lim_{t\to\infty}f(r(t))=0$, which contradicts the lower bound $f(r(\xi_n))\geq C>0$. Therefore, the initial assumption is false, and $\lim_{t\to\infty}r(t)=r^*$.

\qed

\end{document}